\documentclass[11pt]{amsart}
\usepackage[english]{babel}
\usepackage[T1]{fontenc}
\usepackage{geometry}

\usepackage{fullpage}
\usepackage[dvips]{graphicx}
\usepackage{color}

\usepackage{amsmath,amscd,amsthm,amsfonts,latexsym,epsfig}

\theoremstyle{plain}
\newtheorem{thm}{Theorem}[section]
  \newtheorem*{thm*}{Theorem}
 \newtheorem{cor}[thm]{Corollary}
 
 \newtheorem{prop}[thm]{Proposition}
 \theoremstyle{definition}
 
 \theoremstyle{remark}
 \newtheorem{rem}[thm]{Remark}
 \newtheorem{ex}{Example}
 \numberwithin{equation}{section}

\def \Nkt{\mbox{$\mathcal{N}_{\kappa,\tau}$}}
\def \r{\mbox{${\mathbb R}$}}
\def \gt{\mbox{${\widetilde{g}}$}}
\def \N{\mbox{$\mathcal{N}$}}
\def \h{\mbox{${\mathbb H}$}}
\def \E{\mbox{${\mathbb E}$}}
\def \S{\mbox{${\mathbb S}$}}

\def \M{\mbox{$\mathcal{M}$}}

\begin{document}

\title{Bour's theorem and helicoidal surfaces with constant mean curvature in the Bianchi-Cartan-Vranceanu spaces}

\author{Renzo Caddeo}
\address{Universit\`a degli Studi di Cagliari\\
Dipartimento di Matematica e Informatica\\
Via Ospedale 72\\
09124 Cagliari, Italia}

\email{caddeo@unica.it}

\author{Irene I. Onnis}
\address{Universit\`a degli Studi di Cagliari\\
Dipartimento di Matematica e Informatica\\
Via Ospedale 72\\
09124 Cagliari}
\email{irene.onnis@unica.it}

	\author{Paola Piu}\address{Universit\`a degli Studi di Cagliari\\
Dipartimento di Matematica e Informatica\\
Via Ospedale 72\\
09124 Cagliari}
\email{piu@unica.it}

%\date{September 2004}
\subjclass{53A10,  53C40,  53C42}
\keywords{Helicoidal surfaces, constant mean curvature surfaces, BCV-spaces, Bour's theorem.}
\thanks{Work supported by Fondazione di Sardegna (Project STAGE) and Regione Autonoma della Sardegna (Project KASBA)}

\begin{abstract}
In this paper we generalize a classical result of Bour concerning  helicoidal surfaces in the three-dimensional Euclidean space  $\r^3$  to the case of helicoidal surfaces in the Bianchi-Cartan-Vranceanu (BCV) spaces, i.e. in the Riemannian $3$-manifolds whose metrics have groups of isometries of dimension $4$ or $6$, except the hyperbolic one. In particular,  we prove that in a BCV-space  there exists a two-parameter family of helicoidal surfaces isometric to a given helicoidal surface; then, by making use of this two-parameter representation, we characterize helicoidal surfaces which have constant mean curvature, including the minimal ones.
\end{abstract}

\maketitle

\section{Introduction and preliminaries}
Helicoidal surfaces in the Euclidean three dimensional space $\r^3$ are invariant under the action of the $1$-parameter group of helicoidal motions and  are a generalization of rotation surfaces. Since the beginning of differential geometry of surfaces much attention has been given to the surfaces of revolution with constant Gauss curvature or constant mean curvature (CMC-surfaces). The surfaces of revolution with constant Gauss curvature seem to  have been known to Minding (1839, \cite{Mi}), while those with constant mean curvature have been classified by Delaunay (1841, \cite{De}). Helicoidal minimal surfaces were studied by Scherk in 1835 (see \cite{S} and, also, \cite{W}), but it is rather recent the classification of the helicoidal surfaces in $\r^3$ with nonzero constant mean curvature, given by Do Carmo and Dajczer in \cite{DD}. \\

The starting point of their work  in \cite{DD} is a result of Bour about helicoidal surfaces in $\r^3$ (see \cite{bour}, p. 82, Theorem II), for which he received the  mathematics prize awarded by the Acad\'{e}mie des Sciences de Paris in 1861\footnote{The problem that sometimes bears the name of Bour was proposed in 1861 by the  Acad\'emie des Sciences and consists in determining all the surfaces that are isometric to a given surface $(M, ds^2)$.
E. Bour  demonstrated that each helicoidal surface is applicable to a surface of revolution, and that the helices on the first surface correspond to the parallels on the second. Bour's work \cite{bour} contains several theorems on ruled and minimal surfaces; but in its printed version this work does not include  the complete integration of the problem's equations in the case of surfaces of revolution; in fact, it is this  result that enabled Bour to win the Academy's grand prix.}.  Bour proved that there exists a $2$-parameter family of helicoidal surfaces isometric to a given helicoidal surface in $\r^3$. For this, firstly he obtained orthogonal parameters $(u,t)$ on a helicoidal surface $M$ for which the families of $u$-coordinate curves are geodesics on $M$ parametrized by arc length, and the $t$-coordinate curves are the trajectories of the helicoidal motion. Such parameters are called {\it natural parameters} and the first fundamental form with respect to them takes the form  $ds^2=du^2+U^2(u)\,dt^2$. Reciprocally, given the natural parameters $(u,t)$ on $M$  and a function $U(u)$, Bour determined a $2$-parameter family of isometric helicoidal surfaces that have induced metric given by $ds^2=du^2+U^2(u)\,dt^2$, that includes rotation surfaces. An exposition of Bour's results about the theory of deformation of surfaces can be found in the Chapter IX  of \cite{da}. \\

By using the result of Bour, in \cite{DD} Do Carmo and Dajczer established a condition for a surface of the Bour's family to have  constant mean curvature. Also they obtained an integral representation (depending on three parameters) of helicoidal surfaces  with nonzero constant mean curvature, that is a natural generalization of the representation for Delaunay surfaces, i.e. CMC rotation surfaces, given by Kenmotsu (see \cite{Ken}). \\

In  \cite{ET} the authors obtain a generalized Bour's theorem for helicoidal surfaces in the products $\h^2\times\r$ and $\S^2\times\r$, and use it to determine all isometric immersions in these spaces that give the surfaces which are helicoidal and have the same constant mean curvature. \\

In regard to  the study of CMC helicoidal surfaces in BCV-spaces, in \cite{FMP} and in \cite{MO1,Onnis} the authors use the equivariant geometry to classify the profile curves of these surfaces in the Heisenberg group $\h_3$ and in $\h^2\times\r$, respectively.  The case of rotational minimal and constant mean curvature surfaces in the Heisenberg group is treated in \cite{RCPPAR1}. J. Ripoll in \cite{ripoll,ripoll2} classified the CMC invariant surfaces in  the $3$-dimensional sphere $\S^3$ and also in the hyperbolic $3$-space $\h^3$.\\

%In this paper, we establish a Bour's type theorem for helicoidal surfaces in the Bianchi-Cartan-Vranceanu spaces  and we use it to study  the helicoidal surfaces with constant mean curvature, including the minimal ones. These results generalize those obtained in \cite{DD} and \cite{ET}. \\

The aim of this paper is to generalize the results obtained in \cite{DD} and \cite{ET}. The paper is organized as follows. Section~\ref{2} is devoted to give a short description of the Bianchi-Cartan-Vranceanu spaces and the helicoidal surfaces in these spaces. In Section~\ref{3} we establish a Bour's type theorem for helicoidal surfaces in the BCV-spaces (see Theorem~\ref{bour}) and, as an immediate consequence of this result, we have that every helicoidal surface in a BCV-space can be isometrically deformed into a rotation surface through helicoidal surfaces. Moreover,       Corollary~\ref{coro1}   refers to the particular case of isometric rotation surfaces. \\

In Section~\ref{4} we use standard techniques of equivariant geometry, in particular the Reduction Theorem of Back, do Carmo and Hsiang (see \cite{docarmo}), to deduce a differential equation that the function $U(u)$ must  satisfy in order  that a helicoidal surface of the Bour's family determined by $U(u)$
has constant mean curvature. We solve this equation by making a transformation of coordinates, treating separately the case of the space forms $\r^3$ and $\S^3$ from the other BCV-spaces. In this way, we obtain Theorem~\ref{principal} that provides a description, in terms of natural parameters, of all helicoidal surfaces of constant mean curvature in a BCV-space, including the minimal ones. We conclude by showing that in $\r^3$  these results give a natural parametrization of all the helicoidal minimal surfaces obtained by Scherk in \cite{S}.

\section{Helicoidal surfaces in Bianchi-Cartan-Vranceanu spaces}\label{2}
A Riemannian manifold $(\M,g)$ is said to be homogeneous if for every two points $p$ and $q$ in $\M$, there exists an isometry of $\M$, mapping $p$ into $q$.  The classification of $3$-dimensional simply connected  homogeneous spaces is well-known and can be summarized as follows. First of all, the dimension of the isometry group must be equal to $6$, $4$ or $3$ (see \cite{Bi} or \cite{Fu}). Then, if the isometry group is of dimension $6$, $\M$ is a complete real space form, i.e. the Euclidean space $\E^3$, a sphere $\mathbb{S}^3(k)$, or a hyperbolic space $\h^3(k)$. If the dimension of the isometry group is $4$, $\M$ is isometric to $\mathrm{SU}(2)$, the special unitary group, to $\widetilde{\mathrm{SL}(2, R)}$, the universal covering of the real special linear group, to $\mathrm{Nil}_3$, the Heisenberg group, all with a certain left-invariant metric, or to a Riemannian product $\mathbb{S}^2(k) \times \r$ or $\h^2(k) \times \r$. Finally, if the dimension of the isometry group is $3$, $\M$ is also isometric to a simply connected Lie group with a left-invariant metric, for example that called  $\mathrm{SOL}$, one of the Thurston's eight models of geometry \cite{scott}.\\

An explicit classification of $3$-dimensional homogeneous Riemannian metrics based on the dimension of their isometry group was first given by Luigi Bianchi in~1897~(see \cite{Bi} or  \cite{Bi1}). Later \'Elie Cartan in \cite{Ca} and Gheorghe Vranceanu in \cite{Vr} proved that all the metrics whose group of isometries has dimension $4$ or $6$, except the hyperbolic one, can be represented in a  concise form by the following two-parameter family of metrics
\begin{equation}\label{1.1}
g_{\kappa,\tau} =\frac{dx^{2} + dy^{2}}{B^{2}} +  \left(dz +
\tau\, \frac{ydx - xdy}{B}\right)^{2},
\end{equation}

for $\kappa, \tau \in \r$, and  $B=1+\dfrac{\kappa}{4}(x^2+y^2)$, $ (x,y,z) \in\r^3$,  positive.  Thus, {\it the family of metrics $g_{\kappa,\tau}$}, that can rightfully be named the {\bf Bianchi-Cartan-Vranceanu metrics} (BCV metrics)  {\it consists of all three-dimensional homogeneous metrics whose group of isometries has dimension $4$ or $6$, except for those of constant negative sectional curvature}. In the following we shall denote by $\N_{\kappa,\tau}$ the open subset of $\r^3$ where the metrics $g_{\kappa,\tau}$ are defined.\\

With respect to \eqref{1.1} we have the following globally defined orthonormal frame
\begin{equation}\label{eq-basis}
E_1= B\frac{\partial}{\partial x}-\tau y \frac{\partial}{\partial z},\quad
E_2=B\frac{\partial}{\partial y}+\tau x \frac{\partial}{\partial z},\quad
E_3=\frac{\partial}{\partial z}
\end{equation}
and, also, 
\begin{prop}[\cite{piu,PiuProfir}]
 The isometry group of  $g_{\kappa,\tau}$ admits the basis of Killing vector fields
\begin{equation}
	\left\{
\begin{aligned}
&X_1=\Big(1-\frac{\kappa\, y^2}{2 B}\Big)\, E_1+ \frac{\kappa\, x y}{2 B}\, E_2+\frac{2 \tau\, y}{ B}\, E_3, \\
&X_2=\frac{\kappa\, x y}{2 B}\, E_1+ \Big(1-\frac{\kappa\, x^2}{2 B}\Big)\, E_2-\frac{2 \tau\, x}{ B}\, E_3,\\
&X_3=-\frac{y}{ B}\, E_1+ \frac{ x}{B}\, E_2-\frac{\tau\, (x^2+y^2)}{ B}\, E_3,\\
&X_4=E_3.
\end{aligned}
\right.
\end{equation}
\end{prop}

Therefore, the group of isometries of the BCV-spaces contains the helicoidal subgroup, whose infinitesimal generator is the Killing vector field given by
$$
X=-y\frac{\partial}{\partial x}+x\frac{\partial}{\partial y}+a\,\frac{\partial}{\partial z},\qquad a\in\r.
$$
We consider
the surfaces in $\Nkt$ which are invariant under the action of the one-parameter group of isometries $G_X$ of $g_{\kappa,\tau}$ generated by $X$. 
For convenience, we shall introduce cylindrical coordinates  
\begin{equation*}
\left\{
\begin{aligned}
& x=r \cos\theta, \\
& y=r \sin\theta,\\
& z=z,
\end{aligned}
\right.
\end{equation*}
with $r\geq 0$ and $\theta\in(0,2\pi)$. 
In these coordinates the metric \eqref{1.1} becomes 
\begin{equation*}\label{metrica3}
g_{\kappa,\tau} =\frac{dr^{2}}{B^{2}} +r^2\,\Big(\frac{1+\tau^2\, r^2}{B^2}\Big)\,d\theta^2 +dz^2 -2\frac{\tau \,r^2}{B} \,d\theta dz,
\end{equation*}
where $B=1+\dfrac{\kappa}{4}\,r^2$.
Moreover, the  Killing vector field $X$ takes the form
$$X=\frac{\partial}{\partial \theta}+a\,\frac{\partial}{\partial z}$$ and
a set of two invariant functions is
$$\xi_1=r,\qquad
\xi_2=z-a\,\theta.$$
Thus, the orbit space of the action of $G_X$ can be identified with
\begin{equation}\label{OrbitSpace}
\mathcal{B}:=\Nkt/G_X=\{(\xi_1,\xi_2)\in\r^2\;:\;\xi_1\geq 0\}
\end{equation}
and
the orbital distance metric of $\mathcal{B}$ is given by 
	\begin{equation}\label{orbital}
	\gt=  \frac{d\xi_1^2}{B^2}+ \frac{\xi_1^2\,d\xi_2^2}{\xi_1^2+(a\, B-\tau\,\xi_1^2)^2},
	\end{equation}
where $B=1+\dfrac{\kappa}{4}\,\xi_1^2$.\\

Now, consider a helicoidal surface $M$ (with pitch $a$) that, locally, with respect to the cylindrical coordinates, can be parametrized by
\begin{equation}\label{helicoidal}
	\psi(u,\theta)=(\xi_1(u),\theta,\xi_2(u)+a\,\theta),
\end{equation}
and suppose that the profile curve $\tilde{\gamma}(u)=(\xi_1(u),\xi_2(u))$ is parametrized by arc-length in $(\mathcal{B},\gt)$, so that
\begin{equation}\label{ppca}
	\frac{\xi_1'^2}{B^2}+\frac{\xi_1^2\,\xi_2'^2}{\xi_1^2+(a\, B-\tau\,\xi_1^2)^2}=1.
\end{equation}
Therefore from
 \begin{equation*}\label{XpXt_E}
	\left\{
	\begin{aligned}
		&\psi_{u}=\xi_1'\,\Big(\frac{\cos \theta}{B}E_1+\frac{\sin \theta}{B}E_2 \Big) + \xi_2' \,E_3,\\
		& \psi_\theta=\frac{\xi_1}{B}\,(\cos\theta\,E_2-\sin\theta\,E_1)+\bigg(a-\frac{\tau\,\xi_1^2}{B}\bigg)\,E_3=X
	\end{aligned}
	\right.
\end{equation*}
it follows that the coefficients of the induced metric of the helicoidal surface 
are given by

\begin{equation}\label{star}
E(u)=1 + \xi_2'(u)^2\,\bigg(\frac{a\, B(u) - \tau\,\xi_1(u)^2}{B(u)\,\omega(u)}\bigg)^2,\qquad F(u)=\xi_2'(u)\,\bigg(\frac{a\, B(u) - \tau\,\xi_1(u)^2}{B(u)}\bigg),\
\end{equation}

%\begin{equation}\label{star}
%E(u)=\frac{\xi_1'(u)^2}{B(u)^2}+\xi_2'(u)^2,\qquad F(u)=\xi_2'(u)\,\bigg(a-\frac{\tau\,\xi_1(u)^2}{B(u)}\bigg),\
%\end{equation}
and
$$G=\frac{\xi_1(u)^2}{B(u)^2}+\bigg(a-\frac{\tau\,\xi_1(u)^2}{B(u)}\bigg)^2=\omega(u)^2,$$
where $\omega(u)$ is the volume function of the principal orbit.

\section{A Bour's type theorem}\label{3}
In this section, we show that every helicoidal surface in a BCV-space admits a reparametrization by natural parameters and, conversely, given a positive function $U$, it is possible to find a $2$-parameter family of isometric helicoidal surfaces associate with it that are parameterized by natural parameters. 

%This means that the metric, and also the Gaussian curvature, is the same for all the helicoidal surfaces in the family

\begin{thm}\label{bour}
In the BCV-space $\Nkt$  there exists a two parameter family of helicoidal surfaces that are isometric to a given helicoidal surface of the form~\eqref{helicoidal} and that includes a rotation surface. More precisely, for a given positive function $U(u)$ and arbitrary constants  $m\neq 0$ and $a$, the helicoidal surfaces~\eqref{helicoidal} 
whose profile curve $\tilde{\gamma}(u)=(\xi_1(u),\xi_2(u))$ is given by
\begin{equation}\label{xi-bis}
\left\{
\begin{aligned}
\xi_1(u)&=2\sqrt{\frac{m^2\,U^2-a^2}{(1+\sqrt{\Delta})^2-4\tau^2 m^2U^2}},\\
\xi_2(u)&=\int \frac{m\, U\,(4+\kappa\,\xi_1^2)}{4\,\xi_1^2}\,\sqrt{\xi_1^2-\frac{m^4\, U^2\,U'^2\,(4+\kappa\,\xi_1^2)^2}{16\,\Delta}}\,d{u},
\end{aligned}
\right.
\end{equation}
with
\begin{equation}\label{theta-bis}
\theta(u,t)=\frac{t}{m}+\int \frac{(4\,\tau-a\,\kappa)\,\xi_1^2-4 a}{4m\,U\,\xi_1^2}\,\,\sqrt{\xi_1^2-\frac{m^4\, U^2\,U'^2\,(4+\kappa\,\xi_1^2)^2}{16\,\Delta}}\,d{u},
\end{equation}
where
$$\Delta(u)=(1-2a\,\tau)^2+(m^2\,U(u)^2-a^2)(4\tau^2-\kappa),$$
are all to each other isometric and have  first fundamental form given by 
$du^2+U(u)^2\,dt^2$.
\end{thm}
\begin{proof}
From \eqref{star} we have that the induced metric of a helicoidal surface~\eqref{helicoidal}, with pitch $a_0$, is given  by
\begin{equation}\label{metric}
\begin{aligned} g_\psi&=E(u)\,du^2+2\,F(u)\,du\,d\theta+\omega(u)^2\,d\theta^2\\
&=du^2+ \omega(u)^2\,\bigg(d\theta+\xi_2'(u)\,\frac{a_0\, B(u) - \tau\,\xi_1(u)^2}{B(u)\, \omega(u)^2}\,du\bigg)^2,
\end{aligned}
\end{equation}
where
\[
\omega(u)^2 = \frac{\xi_1(u)^2}{B(u)^2}+\bigg(\frac{a_0 \, B(u) -\tau\,\xi_1(u)^2}{B(u)}\bigg)^2 .
\]
Now  we introduce a new parameter $t=t(u,\theta)$ that satisfies: 
\begin{equation}\label{eqgeod-bis}
dt=d\theta+\xi_2'(u)\,\frac{a_0\, B(u) - \tau\,\xi_1(u)^2}{B(u)\, \omega(u)^2}\,du.
\end{equation} 
As the Jacobian $|\partial (u,t)/\partial(u,\theta)|$ is equal to $1$, it follows that $(u,t)$ are local coordinates on a helicoidal surface $M$ and, also, that we can write \eqref{metric} as 
\begin{equation}\label{metric1}
g_\psi=du^2+\omega(u)^2\,dt^2.
\end{equation}
We now observe that the $u$-coordinate curves
%given by $$\alpha(u)=\psi(u,\theta(u,t_0))=\psi\Big(u,t_0-\int  \xi_2'(u)\,\frac{a_0\, B(u) - \tau\,\xi_1(u)^2}{B(u)\, \omega(u)^2}\,du\Big),\qquad t_0\in\r,$$
are parametrized by arc length and also that
%$$F(u)+\omega(u)^2\,\theta'(u,t_0)=0.$$
%Therefore these curves 
are geodesics of $M$ (see \cite{MO2}) which are orthogonal to the $t$-coordinate curves, i.e. the helices. Consequently, the local parametrization $\psi(u,\theta(u,t))$ is a {\it natural} parametrization of the helicoidal surface $M$.\\

Conversely, given a function $U(u)>0$, we want to determine functions $\theta,\xi_1,\xi_2$ of $(u,t)$ such that
\begin{equation}\label{viceversa}
\left\{\begin{aligned}
du^2&=\frac{d\xi_1^2}{B^2}+ \frac{\xi_1^2\,d\xi_2^2}{\xi_1^2+(a\, B-\tau\,\xi_1^2)^2},\\
\pm U(u)\,dt&=\frac{\sqrt{\xi_1^2+(a\,B-\tau\,\xi_1^2)^2}}{B}\,
\bigg[d\theta+\frac{B(a\,B-\tau\,\xi_1^2)}{\xi_1^2+(a\,B-\tau\,\xi_1^2)^2}\,d\xi_2\bigg],
\end{aligned}\right.
\end{equation} 
where $B=1+\dfrac{\kappa}{4}\,\xi_1^2$.\\

From the first equation of \eqref{viceversa} we have that $\xi_i=\xi_i(u)$, $i=1,2$. Then, from the second, we obtain
\begin{equation}\label{viceversa1}
\left\{\begin{aligned}
\frac{\partial \theta}{\partial u}&=-\frac{B(u)\,[a\,B(u)-\tau\,\xi_1^2(u)]}{\xi_1^2(u)+(a\,B(u)-\tau\,\xi_1^2(u))^2}\,\xi_2'(u),\\
\frac{\partial \theta}{\partial t}&=\pm\frac{B(u)\,U(u)}{\sqrt{\xi_1^2(u)+(a\,B(u)-\tau\,\xi_1^2(u))^2}},
\end{aligned}\right.
\end{equation} 
where $B(u)=1+\dfrac{\kappa}{4}\,\xi_1^2(u)$.\\
Therefore, $$\frac{\partial^2 \theta}{\partial t\partial u}=0$$ and hence there exists a constant $m\neq 0$
such that 
\begin{equation}\label{m}
\pm \frac{B(u)\,U(u)}{\sqrt{\xi_1^2(u)+(a\,B(u)-\tau\,\xi_1^2(u))^2}}=\frac{1}{m}.
\end{equation}
Thus the second equation of system~\eqref{viceversa} becomes
\begin{equation}\label{d}
d\theta=\frac{dt}{m}-\frac{B(u)\,(a\,B(u)-\tau\,\xi_1^2(u))}{\xi_1^2(u)+(a\,B(u)-\tau\,\xi_1^2(u))^2}\,d\xi_2.
\end{equation}

If we consider the function $f:=1/\xi_1$,  equation~\eqref{m} can be written as
$$(a^2-m^2\, U^2)\,\|\nabla f\|_{\kappa,\tau} ^2+(1-2 a \,\tau)\,\|\nabla f\|_{\kappa,\tau}+\tau^2-\kappa/4=0$$
and, therefore,
$$\|\nabla f\|_{\kappa,\tau}=\frac{1-2a\,\tau+\sqrt{\Delta}}{2\,(m^2\,U^2-a^2)},$$
with $$\Delta=(1-2a\,\tau)^2+(m^2\,U^2-a^2)(4\tau^2-\kappa).$$
As $\|\nabla f\|_{\kappa,\tau}=f^2+\kappa/4$, we conclude that
\begin{equation}\label{xi1}
\xi_1^2=\frac{4(m^2\,U^2-a^2)}{(1+\sqrt{\Delta})^2-4\tau^2m^2U^2}.
\end{equation}
Then, differentiating \eqref{m} and using \eqref{xi1}, we get
\begin{equation}\label{radice-delta}
m^2\,B^2\,U\,U'=\sqrt{\Delta}\,\xi_1\,\xi_1'
\end{equation}
and hence
\begin{equation}\label{x1B}
\frac{(\xi_1')^2}{B^2}=\frac{m^4\,B^2\, U^2\,U'^2}{\xi_1^2\,\Delta}.
\end{equation}
Therefore,  taking into account the first equation of system~\eqref{viceversa},  we obtain
$$\begin{aligned}
d\xi_2^2&=\frac{m^2\,B^2\,U^2}{\xi_1^2}\,\bigg(1-\frac{\xi_1'^2}{B^2}\bigg)d{u}^2\\
&=\frac{m^2\,B^2\,U^2}{\xi_1^4}\,\bigg(\xi_1^2-\frac{m^4\,B^2\, U^2\,U'^2}{\Delta}\bigg)d{u}^2.
\end{aligned}$$
Thus,  as 
$$
B=1+\frac{\kappa}{4}\,\xi_1^2,
$$
 it turns out that
\begin{equation}\label{xi2}
\xi_2(u)=\int \frac{m\, U\,(4+\kappa\,\xi_1^2)}{4\,\xi_1^2}\,\sqrt{\xi_1^2-\frac{m^4\, U^2\,U'^2\,(4+\kappa\,\xi_1^2)^2}{16\,\Delta}}\,d{u}.
\end{equation}
Also,  from \eqref{d}  we have
\begin{equation}\label{theta}
\theta(u,t)=\frac{t}{m}+\int \frac{(4\,\tau-a\,\kappa)\,\xi_1^2-4 a}{4m\,U\,\xi_1^2}\,\,\sqrt{\xi_1^2-\frac{m^4\, U^2\,U'^2\,(4+\kappa\,\xi_1^2)^2}{16\,\Delta}}\,d{u}.
\end{equation}
Consequently,  the natural parametrization of the helicoidal surface \eqref{helicoidal} with given first
fundamental form $g_\psi=du^2+U(u)^2\,dt^2$  can be calculated by means of equations~\eqref{xi1},  \eqref{xi2} and \eqref{theta}.
\end{proof}

\begin{rem}
If $\kappa=0=\tau$, the BCV-space is the Euclidean space $\r^3$ and  Theorem~\ref{bour} becomes the classical one (\cite{bour}, p.~82, Theorem~II) due to Bour. In fact, in this case the functions $B$ and $\Delta$ are constant and equal to $1$; thus we obtain
\begin{equation}\label{r3}
\left\{\begin{aligned}
\xi_1(u)&=\sqrt{m^2\,U^2-a^2},\\
\xi_2(u)&=\int \frac{m\, U}{m^2\,U^2-a^2}\,\sqrt{m^2\,U^2-a^2-m^4\, U^2\,U'^2}\,d{u},\\
\theta(u,t)&=\frac{t}{m}-\frac{a}{m}\int \frac{\sqrt{m^2\,U^2-a^2-m^4\, U^2\,U'^2}}{U\,(m^2\,U^2-a^2)}\,d{u}.
\end{aligned}
\right.
\end{equation}
\end{rem}

\begin{rem}
The family of helicoidal surfaces $\Psi(u,t):= \psi_{[U,m,a]}(u,t)$  in the BCV-space $\Nkt$ obtained in the Theorem~\ref{bour} depends on two parameters $m\neq 0$ and $a$, and for $m=1$ and $a=a_0$ it contains the original helicoidal surface.  Also, when $m=1$ and $a=0$,  we obtain  a rotational surface isometric to the given helicoidal surface. Therefore, by varying the constant $a$ from $a=0$ to $a=a_0$, we get an isometric deformation from a rotational surface to a given helicoidal surface.
\end{rem}

%\begin{ex}\label{ex1}
%Choosing the function $U(u)=\sqrt{u^2+1}$ and $m=1$ in the formulas~\eqref{r3},  we obtain 
%$$
%\left\{\begin{aligned}
%\xi_1(u)&=\sqrt{u^2+1-a^2},\\
%\xi_2(u)&=\sqrt{1-a^2}\int \frac{\sqrt{u^2+1}}{u^2+1-a^2}\,d{u},\\
%\theta(u,t)&=t-a\, \sqrt{1-a^2}\int \frac{du}{\sqrt{u^2+1}\, (u^2+1-a^2)}.
%\end{aligned}
%\right.
%$$
%By varying $a$ from $a=0$ to $a=1$ we have is the classical isometric deformation of the catenoid $$\psi_{[U,1,0]}(u,t)=(\sqrt{u^2+1}\, \cos t,\sqrt{u^2+1}\,\sin t, \cosh^{-1} (\sqrt{u^2+1}))$$ into the helicoid $\psi_{[U,1,1]}(u,t)=(u \cos t,u \sin t, t+c)$, $c\in\r$, that are minimal surfaces.
%Also, the intermediary helicoidal surfaces are all minimal (see Remark~\ref{rem-m}) and a natural parametrization is given by
%$$\left\{\begin{aligned}
%\xi_1(u)&=\sqrt{u^2+1-a^2},\\
%\xi_2(u)&=\sqrt{1 - a^2}\,\cosh^{-1}(\sqrt{u^2+1} )+ 
%   a\,\arctan \Big(\frac{a\, u}{\sqrt {1 - a^2}\,
%           \sqrt{u^2+1}}\Big),\\
%           \theta(u,t)&=t-a\,\arctan \Big(\frac{a\, u}{\sqrt {1 - a^2}\,
%           \sqrt{u^2+1}}\Big), \qquad 0<a<1. 
%           \end{aligned}
%\right.$$
%\end{ex}

\begin{ex}\label{ex1}
In $\r^3$, we consider the function $U(u)=\sqrt{u^2+d^2}$, $d \in \r$, $d^2 \geq a^2$. If we suppose $m=1$,  from the formulas~\eqref{r3}
we obtain that
$$
\left\{\begin{aligned}
\xi_1(u)&=\sqrt{u^2+d^2-a^2},\\
\xi_2(u)&=\sqrt{d^2-a^2}\int \frac{\sqrt{u^2+d^2}}{u^2+d^2-a^2}\,d{u},\\
\theta(u,t)&=t-a\, \sqrt{d^2-a^2}\int \frac{du}{\sqrt{u^2+d^2}\, [u^2+d^2-a^2]}.
\end{aligned}
\right.
$$
By varying $a$ from $a=0$ to $a=d$ we have the classical isometric deformation of the catenoid $$\psi_{[U,1,0]}(u,t)=\bigg(\sqrt{u^2+d^2}\, \cos t,\sqrt{u^2+d^2}\,\sin t, d\cosh^{-1} \Big(\sqrt{\frac{u^2}{d^2}+1}\Big)\bigg)$$ into the helicoid $\psi_{[U,1,d]}(u,t)=(u \cos t,u \sin t, t+b)$, $b\in\r$, that are minimal surfaces.
Also, the intermediate helicoidal surfaces are all minimal  and  their natural parametrization is given by
$$\left\{\begin{aligned}
\xi_1(u)&=\sqrt{u^2+d^2-a^2},\\
\xi_2(u)&=\sqrt{d^2 -a^2}\,\cosh^{-1} \Big(\sqrt{\frac{u^2}{d^2}+1}\Big)+ 
   a\,\arctan \Big(\frac{a\, u}{\sqrt {d^2 - a^2}\,
           \sqrt{u^2+d^2}}\Big),\\
           \theta(u,t)&=t-\arctan \Big(\frac{a\, u}{\sqrt {d^2 - a^2}\,
           \sqrt{u^2+d^2}}\Big), \qquad 0<a<d. 
           \end{aligned}
\right.$$
Such surfaces are also called second  Scherk's surfaces (see \cite{hi} and, also, the Example~\ref{ex3}). 
\end{ex}

\begin{ex}\label{ex-2}
In the Heisenberg space $\h_3$ equipped with the metric $g_{\kappa,\tau}$ with $\kappa=0$ and $\tau=1/2$, we consider the function $U(u)=(u^2+2)/2$. If we suppose that $m=1$, from  formulas \eqref{xi-bis} we get
\begin{equation}\label{for-cat}
\left\{\begin{aligned}
\xi_1(u)&=\sqrt{\sqrt{u^4+4u^2+8\,(1-a)}+2\,(a-1)},\\
\xi_2(u)&=\int \frac{(2+u^2)}{2\, \xi_1(u)^2}\,\sqrt{\xi_1(u)^2-\frac{u^2\,(u^2+2)^2}{u^4+4u^2+8\,(1-a)}}\,d{u},\\
\theta(u,t)&=t+\int \frac{ \xi_1(u)^2-2 a}{ (u^2+2)\,\xi_1(u)^2}\,\,\sqrt{\xi_1(u)^2-\frac{u^2\,(u^2+2)^2}{u^4+4u^2+8\,(1-a)}}\,d{u}.
\end{aligned}
\right.
\end{equation}
In particular, for $a=1/2$ we obtain the curve
$$
\tilde{\gamma}(u)=(\sqrt{u^2+1},(u+\arctan u) /2),
$$
the profile curve of  the {\em helicoidal catenoid}, that is a helicoidal minimal surface (see \cite{FMP}), parametrized by
$$
\psi(u,\theta)=\Big(\sqrt{u^2+1}\,\cos \theta,\sqrt{u^2+1}\,\sin \theta,\frac{u+\theta+\arctan{u}}{2}\Big).
$$
Also, as $$\theta(u,t)=t-\arctan u+\sqrt{2}\arctan\Big(\frac{u}{\sqrt{2}}\Big),$$ 
we have that  the parametrization
$$\begin{aligned}
\Psi(u,t)&=\psi(u,\theta(u,t))=\bigg(\cos\Big(t+\sqrt{2}\arctan\Big(\frac{u}{\sqrt{2}}\Big)\Big)+u\sin\Big(t+\sqrt{2}\arctan\Big(\frac{u}{\sqrt{2}}\Big)\Big),\\
& \sin\Big(t+\sqrt{2}\arctan\Big(\frac{u}{\sqrt{2}}\Big)\Big)-u\sin\Big(t+\sqrt{2}\arctan\Big(\frac{u}{\sqrt{2}}\Big)\Big),\frac{1}{2} \Big(u+t+\sqrt{2}\arctan\Big(\frac{u}{\sqrt{2}}\Big)
\Big)\bigg)
\end{aligned}$$
represents a natural parametrization of the helicoidal catenoid with
$$g_{\Psi}=du^2+U(u)^2\,dt^2.$$
Now, if we start from $a=1/2$ and in the equations~\eqref{for-cat} we consider all the decreasing values of $a$ in the interval $[0,1/2]$,
  we obtain an isometric deformation of the helicoidal catenoid into a rotational surface (obtained for $a=0$), through helicoidal surfaces  parametrized by natural parameters.

  \begin{figure}[h!]\label{cat-defo}
  \begin{center}
\begin{tabular}{cc}
 \includegraphics[height=1.7in]{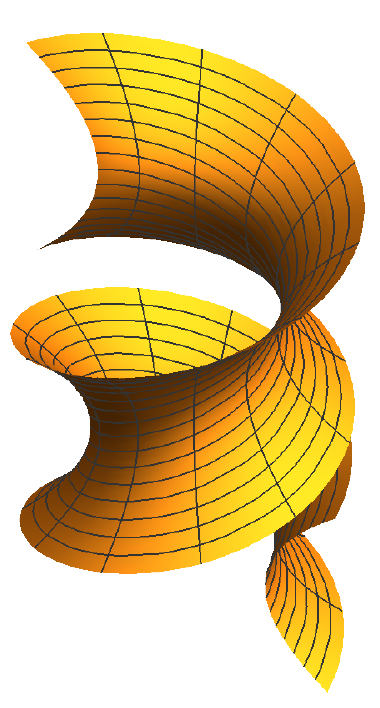} & \hspace{1.5cm}
 \includegraphics[height=1.4in]{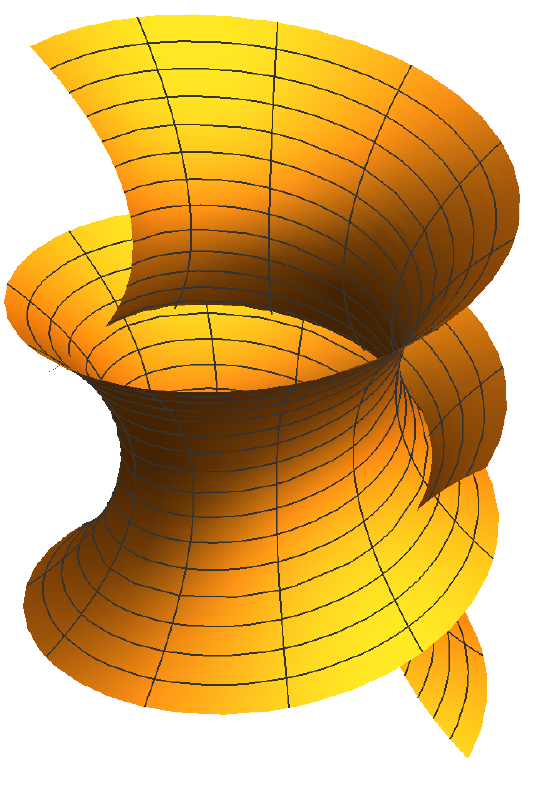} \\
 $a=1/2$ & \hspace{1.5cm} $a=1/4$ \\
 \\
 \includegraphics[height=1.4in]{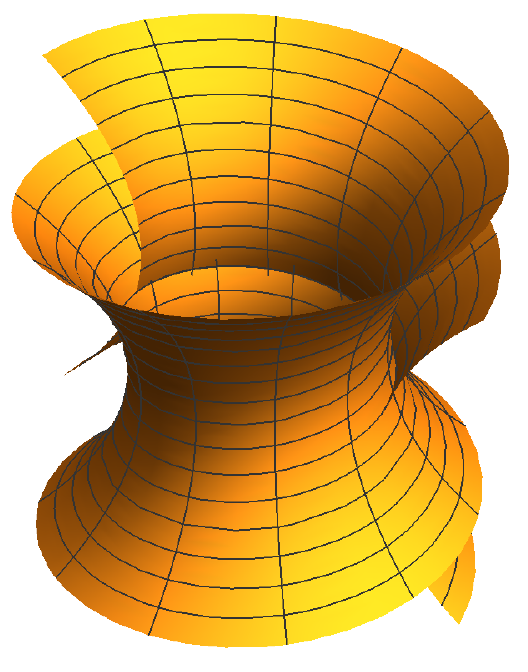} & \hspace{1.5cm}
 \includegraphics[height=1.3in]{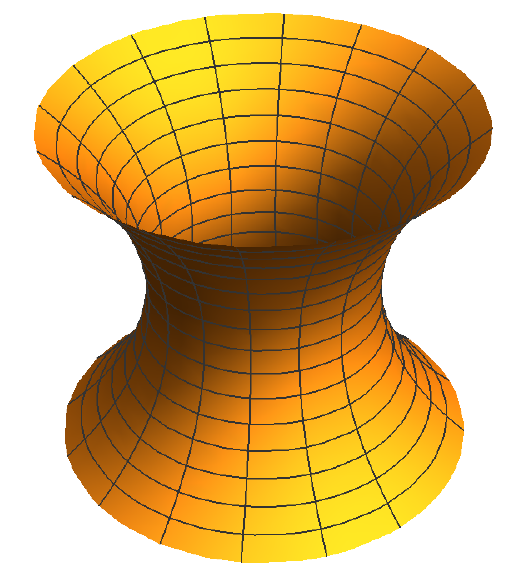} \\
 $a=1/8$ & \hspace{1.5cm} $a=0$ \\ \\
  \end{tabular}
  \caption{Isometric deformation of the helicodal catenoid into a rotation surface in $\h_3$. The surfaces in the picture are obtained for $a=1/2, 1/4, 1/8$ and $a=0$, respectively; only that with angular pitch $a=1/2$ is a minimal surface (see Remark~\ref{rem-b}).}
\end{center}
\end{figure} 

\end{ex}

\begin{cor}\label{coro1}
 The rotation surfaces given by $\psi_{[U,n,0]}(u,t)$, with $n\neq 0$, give rise in the BCV-space $\Nkt$ to  a $1$-parameter family of isometric surfaces  that are also isometric to the helicoidal surfaces $\psi_{[U,m,a]}(u,t)$. This family is determined by the formulas:
\begin{equation}\left\{\begin{aligned}\label{general}
\xi_1(u)&=\frac{2\, n\,U}{\sqrt{2(1+\sqrt{\Delta})-\kappa\, n^2 \,U^2}},\\
\xi_2(u)&=\int\sqrt{\frac{ (1+\sqrt{\Delta})^2}{2(1+\sqrt{\Delta})-\kappa\, n^2\, U^2}-\frac{n^2\,(1+\sqrt{\Delta})^4\,U^2}{\Delta\,[2(1+\sqrt{\Delta})-\kappa\, n^2\, U^2]^2}}\,d{u},\\
\theta(u,t)&=\frac{t}{n}+2\tau\int\sqrt{\frac{1}{2(1+\sqrt{\Delta})-\kappa\, n^2\, U^2}-\frac{n^2\,(1+\sqrt{\Delta})^2\,U'^2}{\Delta\,[2(1+\sqrt{\Delta})-\kappa\, n^2\, U^2]^2}}\,d{u},
\end{aligned}
\right.\end{equation}
where
$\Delta=1+(4\tau^2-\kappa)\,n^2\, U^2$. In particular,
\begin{enumerate}
\item if $\kappa=4\tau^2\neq 0$, i.e. $\Nkt\cong \S^3(\tau^2)$, the equations~\eqref{general} become
$$\left\{\begin{aligned}
\xi_1(u)&=\frac{n\,U}{\sqrt{1-\tau^2\,n^2\, U^2}},\\
\xi_2(u)&=\int\frac{\sqrt{1-n^2\,(\tau^2\, U^2+U'^2)}}{1-\tau^2\,n^2\,U^2}\,d{u},\\
\theta(u,t)&=\frac{t}{n}+\tau\int\frac{\sqrt{1-n^2\,(\tau^2\, U^2+U'^2)}}{1-\tau^2\,n^2\,U^2}\,d{u},
\end{aligned}
\right.$$

\item if $\tau=0$, i.e. $\Nkt\cong \r^3$, $\Nkt\cong \S^2(\kappa)\times\r$ or $\Nkt\cong \h^2(\kappa)\times\r$, the equations~\eqref{general}  become
$$\left\{\begin{aligned}
\xi_1(u)&=\frac{2n\,U}{1+\sqrt{1-\kappa\,n^2\, U^2}},\\
\xi_2(u)&=\int\sqrt{\frac{1-n^2\,(\kappa\, U^2+U'^2)}{1-\kappa\,n^2\,U^2}}\,d{u},\\
\theta(u,t)&=\frac{t}{n},
\end{aligned}
\right.$$
\item if $\kappa =0$ and $\tau\neq 0$, i.e. $\Nkt\cong \h_3$, the equations~\eqref{general} become
$$\left\{\begin{aligned}
\xi_1(u)&=\frac{\sqrt{2}\, n\,U}{\sqrt{1+\sqrt{\Delta}}},\qquad \Delta=1+4\tau^2\,n^2\, U^2,\\
\xi_2(u)&=\int \frac{1+\sqrt{\Delta}}{2}\,\sqrt{\frac{2}{1+\sqrt{\Delta}}-\frac{n^2\,U'^2}{\Delta}}\,d{u},\\
\theta(u,t)&=\frac{t}{n}+\tau\int\sqrt{\frac{2}{1+\sqrt{\Delta}}-\frac{n^2\,U'^2}{\Delta}}\,d{u},
\end{aligned}
\right.$$
\item if $\kappa \neq 0$ and $\tau\neq 0$,  the equations~\eqref{general} give the family of rotation surfaces in the cases $\Nkt\cong\mathrm{SU}(2)$ and $\Nkt\cong\widetilde{\mathrm{SL}(2, \r)}$.
\end{enumerate}
\end{cor}

From the formula for the Gaussian curvature of an invariant surface obtained in \cite{MO}, it follows that  the helicoidal surfaces of the Bour's family in the BCV-space $\Nkt$ have all  the same Gaussian curvature given by
$$K(u)=-\frac{U''(u)}{U(u)}.$$ With regard to the mean curvature $H$ of these surfaces, in  next section we  shall see that  different values  of $a$ and $m$ can give rise to different values of $H$.
\section{Helicoidal surfaces of constant mean curvature}\label{4}
In this section, we will describe the helicoidal surfaces in the BCV-spaces that have the same constant mean curvature.
We start by computing the mean curvature of a helicoidal surface \eqref{helicoidal}. It turns out that the mean curvature of an invariant immersion is tightly  related to the
geodesic curvature of the profile curve, as shown by the remarkable following theorem. But first we recall that 
if  on 
 a three-dimensional connected Riemannian manifold $({N}^3,g)$ we consider   the $1$-parameter subgroup $G_X$ of isometries generated by $X$, an orbit $G(p)$ of $p \in N$ is called {\it principal} if there exists an open neighborhood  $U\subset N$ of $p$ such that all orbits $G(q)$, $q\in U$, are of the same type as $G(p)$ (i.e.  the isotropy subgroups $G_q$ and $G_p$ are coniugated). This implies that $G(q)$  is
diffeomorphic to $G(p)$. We denote with $N_r$ the {\it regular part} of $N$, that is, the subset consisting of points belonging to principal orbits \cite{alexandrino}. Then we have

\begin{thm}[Reduction Theorem \cite{docarmo}]\label{ridu}
Let $H$ be the mean curvature of a $G_X$-invariant surface $M_r\subset N_r$ and $k_g$ the
geodesic curvature of the profile curve $M_r/G_X\subset
\mathcal{B}_r$. Then
$$
H(x)=k_g(\pi(x))-D_{\textbf{n}}\ln \omega(\pi(x)), \qquad x\in M_r,
$$
where $\textbf{n}$ is the unit normal  of  the profile curve and
$\omega=\sqrt{g(X,X)}$ is the volume function of the principal
orbit.
\end{thm}

Let now $\tilde{\gamma}(u)=(\xi_1(u),\xi_2(u))$ be a curve in ${\mathcal B_r}$, parametrized by
arc-length,  that  under the action of $G_{X}$ generates the
helicoidal surface. From \eqref{orbital}, it follows that 
\begin{equation}\label{sigma}
\xi_1'=B\,\cos \sigma,\qquad \xi_2'=\frac{\sqrt{\xi_1^2+(a\, B-\tau\,\xi_1^2)^2}\,\sin \sigma}{\xi_1}
\end{equation}
and the geodesic curvature of $\tilde{\gamma}$ takes the expression
\begin{equation}
\begin{aligned}\label{kg}
k_g=&\frac{(\widetilde{g}_{22})_{\xi_1} \xi_2'-(\widetilde{g}_{11})_{\xi_2} \xi_1' }{2\sqrt{\widetilde{g}_{11}\widetilde{g}_{22}}}+\sigma'\\
&=\frac{B\, [\xi_1^2+(a\, B-\tau\,\xi_1^2)^2]\,(\widetilde{g}_{22})_{\xi_1}}{2\,\xi_1^2}\,\sin\sigma+\sigma',
\end{aligned}
\end{equation}
where $\sigma$ is the angle that
$\tilde{\gamma}$ makes with the $\frac{\partial}{\partial \xi_1}$ direction. Also, as 
$$\textbf{n}=\bigg(-B\,\sin \sigma,\frac{\sqrt{\xi_1^2+(a\, B-\tau\,\xi_1^2)^2}}{\xi_1}\cos\sigma \bigg),$$ the normal derivative is given by
$$D_{\textbf{n}}\ln \omega=\frac{4 \,\xi_1\, [8a \tau-4+ \xi_1^2\,(\kappa + 2 a \kappa\,\tau - 8\tau^2)]\,\sin\sigma}{
   16\,\xi_1^2 + [a\,(4 + \kappa\,\xi_1^2) - 4\tau\,\xi_1^2]^2}$$
and thus we obtain that the mean curvature is given by
\begin{equation}\label{h}
H=\sigma'+\bigg(\frac{1}{\xi_1}-\frac{\kappa}{4}\xi_1\bigg)\sin\sigma.
\end{equation}

\begin{prop}
A helicoidal surface $\Psi(u,t):= \psi_{[U,m,a]}(u,t)$ has constant mean curvature $H$ if and only if $U(u)$ satisfies the differential equation
\begin{equation}\label{eq-principal}
H\,\sqrt{\frac{4\,(m^2\,U^2-a^2)}{(1+\sqrt{\Delta})^2-4\tau^2m^2U^2}-\frac{m^4\,B^2\, U^2\,U'^2}{\Delta}}
=
2-B-m^2\,B\,\Big(\frac{U\,U'}{\sqrt{\Delta}}\bigg)',
\end{equation}
where
\begin{equation}\label{dB}
\Delta=(1-2a\,\tau)^2+(m^2\,U^2-a^2)(4\tau^2-\kappa)\quad \text{and} \quad B=2\frac{1-2 a \tau+\sqrt{\Delta}}{(1+\sqrt{\Delta})^2-4\tau^2m^2 U^2}.\end{equation}
\end{prop}
\begin{proof}
If we consider a helicoidal surface $\Psi(u,t)$ of the Bour's family, from \eqref{sigma}  we can write  equation~\eqref{h} as
 $$
H=-\frac{\big(\frac{\xi_1'}{B}\big)'}{\sqrt{1-\big(\frac{\xi_1'}{B}\big)^2}}+\bigg(\frac{1}{\xi_1}-\frac{\kappa}{4}\xi_1\bigg)\,\sqrt{1-\Big(\frac{\xi_1'}{B}\Big)^2}
$$
and therefore, using \eqref{x1B} we get
\begin{equation}\label{hbis}
\begin{aligned}
&-\Big(\frac{\xi_1'}{B}\Big)'+\bigg(\frac{1}{\xi_1}-\frac{\kappa}{4}\xi_1\bigg)\,\Big[1-\Big(\frac{\xi_1'}{B}\Big)^2\Big]=H\sqrt{1-\Big(\frac{\xi_1'}{B}\Big)^2}\\&=\frac{H}{\xi_1}\sqrt{\xi_1^2-\frac{m^4\,B^2\, U^2\,U'^2}{\Delta}}.
\end{aligned}
\end{equation}
Then, 
\begin{equation}\label{radice-Q}
\begin{aligned}
&H\,\sqrt{\xi_1^2-\frac{m^4\,B^2\, U^2\,U'^2}{\Delta}}\\&=
\xi_1\,\bigg[-\Big(\frac{\xi_1'}{B}\Big)'+\bigg(\frac{1}{\xi_1}-\frac{\kappa}{4}\xi_1\bigg)\,\Big(1-\Big(\frac{\xi_1'}{B}\Big)^2\Big)\bigg]\\&=
1-\Big(\frac{\xi_1'}{B}\Big)^2-\frac{\xi_1\,\xi_1''}{B}+\frac{3\kappa\,\xi_1^2}{4}\Big(\frac{\xi_1'}{B}\Big)^2-\frac{k}{4}\xi_1^2.
\end{aligned}
\end{equation}
Now, deriving \eqref{radice-delta} we get
$$\xi_1\,\xi_1''=\frac{m^2\,B^2}{\sqrt{\Delta}}\,(U'^2+U\,U'')+\frac{m^4\,B^2}{\Delta}\bigg(\kappa\,B+\frac{\kappa-4\tau^2}{\sqrt{\Delta}}\bigg)\,U^2\,U'^2-\xi_1'^2$$
and hence, taking into account \eqref{xi1}, \eqref{radice-delta} and \eqref{x1B}, we can write equation~\eqref{radice-Q} as
$$\begin{aligned}
&H\,\sqrt{\xi_1^2-\frac{m^4\,B^2\, U^2\,U'^2}{\Delta}}
\\&=
2-B-\frac{m^2\,B}{\sqrt{\Delta}}\,(U'^2+U\,U'')+\frac{m^4\,B}{\Delta}\,\bigg(\frac{4\tau^2-\kappa}{\sqrt{\Delta}}-\frac{\kappa\,B}{4}+\frac{B^2-B}{\xi_1^2}\bigg)U^2\,U'^2
\\&=
2-B-\frac{m^2\,B}{\sqrt{\Delta}}\,(U'^2+U\,U'')+\frac{(4\tau^2-\kappa)\,m^4\,B}{\Delta^{3/2}}\,U^2\,U'^2.
\end{aligned}
$$

\end{proof}

\subsection{The solution of the mean curvature equation}
Next, we will give a description of the helicoidal surfaces in $\Nkt$ with constant mean curvature $H$. For this purpose, we assume that the helicoidal surfaces are parametrized by natural coordinates $(u,t)$ and we determine explicitly the expression of the function $U(u)$, that gives the metric, by integrating~\eqref{eq-principal}.
\begin{thm}\label{principal}
In the BCV-space $\Nkt$ the helicoidal surface $\Psi(u,t):=(\xi_1(u),\theta(u,t),\xi_2(u)+a\,\theta(u,t))$  with $\xi_1(u), \xi_2(u)$ and $\theta(u,t)$ given by \eqref{xi-bis} and \eqref{theta-bis}, has constant mean curvature $H$ if and only if
$U(u)$ is given by:

\begin{enumerate}
\item if $\kappa=\tau=H=0$, $$ U^2(u)=\frac{u^2+a^2+c^2/4}{m^2};$$
\item if $\kappa=4\tau^2\neq -H^2$,
$$U^2(u)=\frac{c_1+
\sqrt{c_1^2+c_2\, (H^2+4\tau^2)}\,\sin (\sqrt{4\tau^2+H^2}\,u)}{m^2(H^2+4\tau^2)};$$
\item  if $-H^2=\kappa\neq 4\tau^2$,
$$U^2(u)=\frac{\Big(\dfrac{b_1}{2}u^2+b_2\Big)^2+b_3}{m^2\,(4\tau^2+H^2)};$$
\item if $-H^2<\kappa\neq 4\tau^2$,
$$U^2(u)=\frac{(H^2+\kappa)^2\,b_3+
\Big[b_1+\sqrt{b_1^2+b\,(H^2+\kappa)}\,\sin(\sqrt{H^2+\kappa}\, u)\Big]^2}{m^2(4\tau^2-\kappa)(H^2+\kappa)^2};$$

\item if $-H^2>\kappa\neq 4\tau^2$,
$$
U^2(u)=\left\{\begin{aligned}
&\frac{(H^2+\kappa)^2\, b_3+
\Big[b_1-\sqrt{-b_1^2-b\,(H^2+\kappa)}\,\sinh(\sqrt{-(H^2+\kappa)}\, u)\Big]^2}{m^2(4\tau^2-\kappa)(H^2+\kappa)^2}, \quad b_1^2+b\,(H^2+\kappa)<0,\\ 
& \frac{(H^2+\kappa)^2\, b_3+
\Big[b_1-\sqrt{b_1^2+b\,(H^2+\kappa)}\,\cosh(\sqrt{-(H^2+\kappa)}\, u)\Big]^2}{m^2(4\tau^2-\kappa)(H^2+\kappa)^2},\qquad b_1^2+b\,(H^2+\kappa)>0,
\end{aligned}
\right.
$$
where  $b, b_i, c_i\in\r$, are the constants given by
$$\begin{aligned}
b&=(1-2a\tau)\,[\kappa\,(1+2 a \tau)-8\tau^2]-c^2, \quad & b_1&=4\tau^2-2a\kappa\tau-c \,H, \\ 
b_2&=-\frac{b}{2\,b_1},\qquad & b_3&=4a\tau-a^2\kappa-1, &\\
c_1&= 1+ (1-2a\,\tau)^2-c\, H,  \quad & c_2&=-c^2-4 a^2\,(1-a\,\tau).                                              
\end{aligned}$$
\end{enumerate}
These expressions define a $1$-parameter family $\{U_c(u)\}$ of functions $U(u)$ such that the helicoidal surface $\psi_{[U_c,m,a]}(u,t)$ has constant mean curvature and equal to $H$.
\end{thm}
\begin{proof}
Using the transformation of coordinates given by 
\begin{equation}\label{tc}
\left\{
\begin{aligned}
x(u)&=m\,U(u),\\
 y(u)&=\sqrt{\frac{(x^2(u)-a^2)\,\,[(1+\sqrt{\Delta(u)})^2-4\tau^2\,x^2(u)]}{[1-2 a \tau+\sqrt{\Delta(u)}]^2}-\frac{x(u)^2\,x'(u)^2}{\Delta(u)}},
 \end{aligned}
 \right.
\end{equation}
where 
$$
\Delta(u)=(1-2a\,\tau)^2+(4\tau^2-\kappa)(x^2(u)-a^2),$$
equation~\eqref{eq-principal} becomes
\begin{equation}\label{aux}
H\,y(u)=\frac{2}{B(u)}-1-\Bigg(\frac{x(u)\,x'(u)}{\sqrt{\Delta(u)}}\Bigg)',
\end{equation}
with
$$B(u)=2\frac{1-2 a \tau+\sqrt{\Delta(u)}}{(1+\sqrt{\Delta(u)})^2-4\tau^2\,x^2(u)}.$$
Therefore,
$$y'(u)=\frac{x(u)\,x'(u)}{y(u)\,\sqrt{\Delta(u)}}\bigg[\frac{2}{B(u)}-1-\Bigg(\frac{x(u)\,x'(u)}{\sqrt{\Delta(u)}}\Bigg)'\bigg]=\frac{H\,x(u)\,x'(u)}{\sqrt{\Delta(u)}}$$
and thus
\begin{equation}\label{solution}
y(u)=\left\{
\begin{aligned}
& \frac{H\,x^2(u)+c}{2\,\sqrt{\Delta(u)}}, \quad \qquad\;\; 4\tau^2-\kappa= 0,\\
&\frac{H\sqrt{\Delta(u)}+c}{4\tau^2-\kappa},\qquad \quad 4\tau^2-\kappa\neq 0,
\end{aligned}
\right.
\end{equation}
where  $c$ is an arbitrary constant. Consequently, we have the following cases:

\subsection*{Case $4\tau^2-\kappa= 0$}
From \eqref{tc} and \eqref{solution}, if we suppose that $1-2a\,\tau>0$, then we have
$$
\big(2\, x(u)\, x'(u)\big)^2  =
4 \big[(1-2a\,\tau) -\tau^2\,(x^2(u)-a^2)\big]\,(x^2(u)-a^2)-(H\, x^2(u)+c)^2.
$$
Putting $z=x^2$ the above expression is transformed into the following
\begin{equation}\label{eq-z}z'(u)=\sqrt{-(H^2+4\tau^2)\,z^2(u)+2c_1\,z(u)+c_2},
\end{equation}
where $$c_1= 1+ (1-2a\,\tau)^2-c\, H,  \qquad c_2=-c^2-4 a^2\,(1-a\,\tau).$$
Consequently, 

\begin{itemize}
\item[i)] If $H^2+4\tau^2=0$, we have that the BCV-space is $\r^3$ and $c_1=2$, $c_2=-c^2-4a^2$. Thus, adjusting the origin of $u$, we get 
$$\sqrt{z+\frac{c_2}{4}}=u$$
and since $z=m^2\,U^2$ we have
\begin{equation}\label{U1}
U^2(u)=\frac{u^2+a^2+c^2/4}{m^2}.
\end{equation}
We observe that when $c=0$, from \eqref{r3} it follows that the helicoidal surface $\psi_{[U,m,a]}$ is a helicoid.
\item[ii)] If $H^2+4\tau^2\neq 0$, by integrating \eqref{eq-z} we have that
$$u=\frac{1}{\sqrt{H^2+4\tau^2}}\sin^{-1}\bigg(\frac{(H^2+4\tau^2)\,z-c_1}{\sqrt{c_1^2+c_2\,(H^2+4\tau^2)}}\bigg),$$
%$$u=\frac{1}{\sqrt{H^2+4\tau^2}}\sin^{-1}\bigg(\frac{(H^2+4\tau^2)\,z+c\, H-1-(1-2 a \tau)^2}{2\sqrt{(1-2a\tau)^2-H\,(1-2a\tau) (Ha^2+c)-\tau^2\, (Ha^2+c)^2}}\bigg),$$
up to  a constant. Thus, as $z=m^2\,U^2$, it follows that
\begin{equation}\label{U2}\begin{aligned}
U^2(u)&=\frac{c_1+
\sqrt{c_1^2+c_2\, (H^2+4\tau^2)}\,\sin (\sqrt{4\tau^2+H^2}\,u)}{m^2(H^2+4\tau^2)}
\\&=\frac{1}{m^2(H^2+4\tau^2)}\,\bigg[1+ (1-2a\,\tau)^2-c\, H\\&+
2\sqrt{(1-2a\tau)^2-H\,(1-2a\tau) (Ha^2+c)-\tau^2\, (Ha^2+c)^2}\,\sin (\sqrt{H^2+4\tau^2}\,u)\bigg].
\end{aligned}
\end{equation}
\end{itemize}

\subsection*{Case $4\tau^2-\kappa\neq 0$}
In this case, as 
$$\begin{aligned} &x^2(u)=\frac{1+a^2\kappa-4a\,\tau-\Delta (u)}{\kappa-4\tau^2}, \\
 &\frac{x(u)\,x'(u)}{\sqrt{\Delta(u)}}=\frac{(\sqrt{\Delta(u)})'}{4\tau^2-\kappa},
 \end{aligned}$$ 
 from \eqref{tc} and \eqref{solution} we get
$$\begin{aligned}
&(H\,\sqrt{\Delta(u)}+c)^2=(4\tau^2-\kappa)^2\,y^2(u)\\
&=(1-2a\tau-\sqrt{\Delta(u)})\,(\kappa+\kappa\sqrt{\Delta(u)}+2 a \kappa\tau-8\tau^2)-\Big((\sqrt{\Delta(u)})'\Big)^2.
\end{aligned}$$
Thus
\begin{equation}\label{eq-d-delta}
(\sqrt{\Delta(u)})'=\sqrt{-(H^2+\kappa)\,\Delta(u)+2\,b_1\sqrt{\Delta(u)}+b},
\end{equation}
where $$b=(1-2a\tau)\,(\kappa\,(1+2 a \tau)-8\tau^2)-c^2,\qquad b_1=4\tau^2-2a\kappa\tau-c \,H.$$
In the sequel we integrate equation~\eqref{eq-d-delta}, up to a change of the origin of $u$, considering the following possibilities:
\begin{itemize}
\item[i)]  If  $H^2+\kappa=0$, then we obtain that
$$\sqrt{\Delta(u)}=\frac{b_1}{2}u^2+b_2,$$
with
$$b_1=2a\tau \,H^2+4\tau^2-c\,H,\qquad b_2=-\frac{b}{2\,b_1}.$$
Then, substituting in the first equation of \eqref{dB}, we get
\begin{equation}\label{U3}
\begin{aligned}
U^2(u)&=\frac{1}{m^2}\bigg(\frac{\Delta(u)-(1-2a\tau)^2}{4\tau^2+H^2}+a^2\bigg)\\
&=\frac{\Big(\dfrac{b_1}{2}u^2+b_2\Big)^2+a^2H^2+4a\tau-1}{m^2\,(4\tau^2+H^2)}.
\end{aligned}
\end{equation}
\item[ii)] If $H^2+\kappa >0$, then the integration of \eqref{eq-d-delta} gives
$$\sqrt{\Delta(u)}=\frac{1}{H^2+\kappa}\,\Big[b_1+
\sqrt{b_1^2+b\,(H^2+\kappa)}\,\sin(\sqrt{H^2+\kappa}\, u)\Big].$$
Therefore, substituting in the first equation of \eqref{dB}, we obtain 
\begin{equation}\label{U4}
U^2(u)=\frac{(H^2+\kappa)^2(4a\tau-a^2\kappa-1)+
\Big[b_1+\sqrt{b_1^2+b\,(H^2+\kappa)}\,\sin(\sqrt{H^2+\kappa}\, u)\Big]^2}{m^2(4\tau^2-\kappa)(H^2+\kappa)^2}.
\end{equation}
\item[iii)] If $H^2+\kappa <0$, then the integration of \eqref{eq-d-delta} gives
$$\sqrt{\Delta(u)}=\left\{\begin{aligned}
&\frac{1}{H^2+\kappa}\,\Big[b_1-\sqrt{-b_1^2-b\,(H^2+\kappa)}\,\sinh(\sqrt{-(H^2+\kappa)}\, u)\Big],\quad b_1^2+b\,(H^2+\kappa)<0,\\
&\frac{1}{H^2+\kappa}\,\Big[b_1-\sqrt{b_1^2+b\,(H^2+\kappa)}\,\cosh(\sqrt{-(H^2+\kappa)}\, u)\Big],
\qquad b_1^2+b\,(H^2+\kappa)>0.
\end{aligned}
\right.
$$
Therefore, substituting in the first equation of \eqref{dB}, we obtain  
\begin{equation}\label{U5}
U^2(u)=\left\{\begin{aligned}
&\frac{(H^2+\kappa)^2\, b_3+
\Big[b_1-\sqrt{-b_1^2-b\,(H^2+\kappa)}\,\sinh(\sqrt{-(H^2+\kappa)}\, u)\Big]^2}{m^2(4\tau^2-\kappa)(H^2+\kappa)^2}, \quad b_1^2+b\,(H^2+\kappa)<0,\\ 
& \frac{(H^2+\kappa)^2\, b_3+
\Big[b_1-\sqrt{b_1^2+b\,(H^2+\kappa)}\,\cosh(\sqrt{-(H^2+\kappa)}\, u)\Big]^2}{m^2(4\tau^2-\kappa)(H^2+\kappa)^2},\qquad b_1^2+b\,(H^2+\kappa)>0.
\end{aligned}
\right.
\end{equation}
\end{itemize}
\end{proof}
\begin{rem}
In particular, if we consider $m=1$ and $a=a_0$ in the expressions of $U(u)$ obtained in Theorem~\ref{principal}, we see that an arbitrary helicoidal surface in $\Nkt$ has constant mean curvature $H$ if and only if the functions $\xi_1(u), \xi_2(u)$ and $\theta(u,t)$ are given by \eqref{xi-bis} and \eqref{theta-bis} by substituting  the corresponding function $U(u)$.
\end{rem}
\begin{rem}
Putting $\kappa=\tau=0$ in the equations~\eqref{U1} and \eqref{U2} we obtain the following expressions:
$$U^2(u)=\left\{
\begin{aligned}
&\frac{u^2+a^2+c^2/4}{m^2},\qquad \text{if}\quad H=0,\\
&\frac{2-c\,H+2\sqrt{1-c\,H-a^2\,H^2}\sin{(H\,u)}}{m^2\,H^2},\quad \text{if}\quad H\neq 0,
\end{aligned}
\right.$$
the second of which was given by Do Carmo and Dajczer in \cite{DD}.
\end{rem}
In regard to the helicoidal minimal surfaces we have the following result:
\begin{cor}\label{coro}
In  each BCV-space $\Nkt$, the helicoidal minimal surfaces $\Psi(u,t):=(\xi_1(u),\theta(u,t),\xi_2(u)+a\,\theta(u,t))$, where $\xi_1(u), \xi_2(u)$ and $\theta(u,t)$ are given by \eqref{xi-bis} and \eqref{theta-bis},  are determined 
by the following functions $U(u)$:
\begin{enumerate}
\item If $\kappa=\tau=0$, then $\Nkt\cong \r^3$ and 
\begin{equation}\label{u-minimal}
 U^2(u)=\frac{u^2+a^2+c^2/4}{m^2};
 \end{equation}
\item if $\kappa=4\tau^2\neq 0$, then $\Nkt\cong \S^3(\tau^2)$ and 
$$U^2(u)=\frac{1-a\,\tau+\sqrt{(1-2a\tau)^2-\tau^2\,c^2}\sin(2\tau u)}{4 m^2\tau^2},$$
with $|c|<|1/\tau-2a|$;
\item if $\kappa >0$ and $\tau=0$, then $\Nkt\cong \S^2(\kappa)\times\r$ and
$$U^2(u)=\frac{\kappa\,(a^2\kappa+1)+(c^2-\kappa)\,\sin^2(\sqrt{\kappa}\, u)}{m^2\,\kappa^2},\qquad |c|<\sqrt{\kappa};$$
\item if $\kappa <0$ and $\tau=0$, then $\Nkt\cong \h^2(\kappa)\times\r$
and
$$U^2(u)=\frac{\kappa\,(a^2\kappa+1)+(c^2-\kappa)\,\cosh^2(\sqrt{-\kappa}\, u)}{m^2\,\kappa^2};$$
\item if $\kappa =0$ and $\tau\neq 0$, then $\Nkt\cong \h_3$ and
\begin{equation}\label{u-minimal-2}
U^2(u)=\frac{\Big(2\,\tau^2 u^2+1-2a\tau+\dfrac{c^2}{8\tau^2}\Big)^2+4a\,\tau -1}{4m^2\,\tau^2},
\end{equation}
\item if $\kappa >0$ and $\tau\neq 0$, then $\Nkt\cong\mathrm{SU}(2)$ and
$$U^2(u)=\frac{\kappa^2\,(4a\tau-a^2\kappa-1)+\Big[4\tau^2-2 a \kappa\tau+\sqrt{(4\tau^2-\kappa)^2-c^2\kappa}\,\sin(\sqrt{\kappa}\,u)\Big]^2}{m^2\kappa^2\,(4\tau^2-\kappa)},$$
with $|c|<|4\tau^2-\kappa|/\sqrt{k}$;
\item if $\kappa <0$ and $\tau\neq 0$, then $\Nkt\cong\widetilde{\mathrm{SL}(2, \r)}$ and
$$U^2(u)=\frac{\kappa^2\,(4a\tau-a^2\kappa-1)+\Big[4\tau^2-2 a \kappa\tau-\sqrt{(4\tau^2-\kappa)^2-c^2\kappa}\,\cosh(\sqrt{-\kappa}\,u)\Big]^2}{m^2\kappa^2\,(4\tau^2-\kappa)}.$$
\end{enumerate}
\end{cor}

%\begin{rem}\label{rem-m}
%The helicoidal surfaces obtained in the Example~\ref{ex1} for $0\leq a\leq 1$ and $m=1$ are all minimal surfaces in $\r^3$ because the function $U^2(u)=u^2+1$ is 
%obtained choosing $c_3=0$ and $c^2=4 (1-a^2)$ in the equation~\eqref{u-minimal}.
%\end{rem}

\begin{ex}[General helicoidal minimal surface in $\r^3$]\label{ex3}
In $\r^3$ we consider the function \eqref{u-minimal} for $m=1$ replacing,  for semplicity,  the constant $c/2$ by $c$.  From the formulas~\eqref{r3}
we obtain that the natural parametrization of a general helicoidal minimal surface in the euclidean space
 is given by
\begin{equation}\label{scherk}
\left\{\begin{aligned}
\xi_1(u)&=\sqrt{u^2+c^2},\\
\xi_2(u)&=c\,\cosh^{-1} \Big(\sqrt{\frac{u^2}{a^2+c^2}+1}\Big)+ 
   a\,\arctan \Big(\frac{a\, u}{c\,
           \sqrt{u^2+a^2+c^2}}\Big),
           \\      \theta(u,t)&=t-\arctan \Big(\frac{a\, u}{c\,
       \sqrt{u^2+a^2+c^2}}\Big).
           \end{aligned}
\right.
\end{equation}
This family of surfaces  called second Scherk's surfaces\footnote{In 1835 H.F.~Scherk made an important contribution to  minimal surfaces theory with his work \cite{S} that contains the first examples of minimal surfaces obtained from the integral of Monge and Legendre. Also, he investigated   minimality of surfaces given as graphs $z=z(r,\theta)$ (where $(r,\theta,z)$ are cylindrical coordinates in $\r^3$)  satisfying  the condition $\frac{\partial^2 z}{\partial r\partial\theta}=0$, and determined all the helicoidal minimal surfaces. Detailed accounts and further information can be found in \cite{clapier}, on p. ~60, and in \cite{da}, on p. ~327.}, 
includes the catenoid and the helicoid, that correspond to the cases $a=0$ and $c=0$, respectively. We observe that all the surfaces of the family for which the sum $a^2+c^2$ is the same are isometric to each other. Also, from Example~\ref{ex1},  it follows that every helicoidal minimal surface in $\r^3$ belongs to one of the families of isometric surfaces obtained deforming catenoids into helicoids.
\end{ex}

\begin{rem}\label{rem-b}
Among the helicoidal surfaces in $\h_3$ obtained in   Example~\ref{ex-2} the only minimal surface is the helicoidal catenoid because the function $U(u)=(u^2+2)/2$ can be obtained just choosing in \eqref{u-minimal-2} $c^2=1=m$ and $a=1/2$.
\end{rem}


\begin{thebibliography}{10}
\bibitem{alexandrino} M.M.~Alexandrino, R.G.~Bettiol.
{\it Lie groups and geometric aspects of isometric actions}. Springer, Cham, 2015.

\bibitem{docarmo}  A. Back, M.P. do Carmo, W.Y. Hsiang. 
On some fundamental equations of equivariant Riemannian geometry. 
 {\it Tamkang J. Math.} 40 (2009), no. 4, 343--376.

%\bibitem{BK} C. Baikoussis, T.  Koufogiorgos. Helicoidal surfaces with prescribed mean or Gaussian curvature.
%{\it J. Geom.} 63 (1998), 25--29.

\bibitem{Bi} L.~Bianchi. {\it Sugli spazi a tre dimensioni che ammettono un gruppo continuo di movimenti,} Mem. Soc. It. delle Scienze (dei XL) (b), 11, (1897) pag. 267-352. 

\bibitem{Bi1} L.~Bianchi. {\it Gruppi continui e finiti}. Ed. Zanichelli,
Bologna, 1928.

\bibitem{bour} E. Bour.  Memoire sur le deformation de surfaces. {\it Journal de l'\'{E}cole Polytechnique}, XXXIX Cahier, (1862), 1--148.


\bibitem{RCPPAR1} R.~Caddeo, P.~Piu, A.~Ratto. 
 ${SO}(2)$-invariant minimal and constant mean curvature surfaces in three dimensional homogeneous spaces. 
 {\em Manuscripta Math.} 87 (1995), 1--12.

%\bibitem{RCPPAR2} R.~Caddeo, P.~Piu, A.~Ratto. 
% Rotational surfaces in $\Bbb{H}_{3}$ with constant Gauss curvature. 
% {\em  Boll. Un. Mat. Ital. B}  10 (1996), 341--357.

\bibitem{Ca} \'{E}.~Cartan. {\em Le\c{c}ons sur la g\'{e}om\'{e}trie
des espaces de Riemann}. Gauthier Villars, Paris, 1946.

\bibitem{clapier} F.C. Clapier. Sur les surfaces minima ou \'{e}lassoides.
Th\`{e}ses de l'entre-deux-guerres, 1919.

\bibitem{da} 
G. Darboux. {\em Le\c{c}ons Sur la Th\'eorie G\'en\'erale des Surfaces}.
Vol. I, Paris, 1914.

\bibitem{De}
C. Delaunay. Sur la surface de r\'{e}volution dont la courbure moyenne est constante,
 {\em J. Math. Pures et Appl.} 1 (1841), 309--320.
 
 \bibitem{hi} U. Dierkes, S. Hildebrandt, F. Sauvigny. {\em Minimal surfaces I},  Springer, Heidelberg, 2010.

\bibitem{DD} M.P. do Carmo, M. Dajczer.
 Helicoidal surfaces with constant mean curvature. 
{\em T\^ohoku Math. J.} 34 (1982), 425--435.



\bibitem{FMP} C.B.~Figueroa, F.~Mercuri, R.H.L.~Pedrosa. 
Invariant surfaces of the Heisenberg groups.
{\em Ann. Mat. Pura Appl.} 177 (1999), 173--194.

\bibitem{Fu}
G. Fubini.
Sugli spazi che ammettono un gruppo continuo di
movimenti,
{\em Ann. di Matem.} Tomo 8, serie III, (1903), 39--82.


\bibitem{Ken} K. Kenmotsu. 
 Surfaces of revolution with prescribed mean curvature.
 {\em T\^ohoku
Math. J.} 32 (1980), 147--153. 

%\bibitem{LO} R.~Lopez, M.I. Munteanu. 
%\newblock Invariant surfaces in homogenous space Sol with constant curvature.
%\newblock {\em  Mathematische Nachrichten} 287 (2014), 1013--1024.

\bibitem{Mi} F.A. Minding.
 Wie sich unterscheiden l\"a{\ss}t,
ob zwei gegebene krumme F\"achen aufeinander abwickelbar sind oder
nicht; nebst Bemerkungen \"uber F\"achen von
unver\"anderlichem Kr\"ummungsma\ss,
{\em J. reine angew. Math.} ({\it Giornale di Crelle}) {\bf 19} (1839), 370--387.


\bibitem{MO1}
S. Montaldo, I.I. Onnis.
Invariant CMC surfaces in $\Bbb H^2\times\Bbb R$.
{\em  Glasg. Math. J.}  46  (2004), 311--321.

\bibitem{MO}
S. Montaldo, I.I. Onnis.
Invariant surfaces in a three-manifold with constant
Gaussian curvature.
 {\em J. Geom. Phys.} 55  (2005), 440--449.

%\bibitem{MO3} S. Montaldo, I.I. Onnis.  
%\newblock Biharmonic curves on an invariant surfaces,
%\newblock  {\em J. Geom. Phys.} (59) 3 (2009), 391--399
%
\bibitem{MO2} S. Montaldo, I.I. Onnis.  
\newblock Geodesics on an invariant surface,  
\newblock {\em J. Geom. Phys.} 61 (2011), 1385--1395.


%\bibitem{Olver}
%P.J.~Olver.
%\newblock {\em Application of Lie Groups to Differential Equations}.
%\newblock  GTM 107, Springer-Verlag, New York, 1986.

\bibitem{Onnis} I.I.~Onnis.
Invariant surfaces with constant mean curvature in $\h^2\times\r$.
{\em Ann. Mat. Pura Appl.} 187 (2008), 667--682.

%\bibitem{palais} R.S. ~Palais.
%\newblock  On the existence of slices for actions of non-compact Lie groups.
%\newblock {\em  Ann. of Math.} (2) 73 (1961),  295--323.

 \bibitem{piu} P.~Piu, Sur les flots riemanniens des espaces de D'Atri de dimension 3.
{\em Rend. Sem. Mat. Univ. Politec. Torino} 46 (1988), 171--187.
% 
% \bibitem{Piu} M.P.~Piu. {\em Sur certains types de distributions non-integrables totalement g\'eod\'esiques}. Th\'ese de Doctorat, Universit\'e de Haute-Alsace (1988).
 
 \bibitem{PiuProfir} P.~Piu, M.M.~Profir. On the three-dimensional homogenous SO(2)-isotropic Riemannian manifolds. {\em An. Stiint. Univ. Al. I. Cuza Iasi. Mat. (N.S.)} 57 (2011), 361--376. 
 

\bibitem{ripoll} J.B. Ripoll. Superf\'icies invariantes de curvatura m\'edia constante em $Q^3$, Tese de Doutorado, IMPA, 1986.

\bibitem{ripoll2} J.B. Ripoll. Helicoidal minimal surfaces in hyperbolic space. 
{\em Nagoya Math. J.}
114 (1989), 65--75.

 
% \bibitem{earp}
% R. S\'{a} Earp.
%Parabolic and hyperbolic screw motion surfaces in
%$\mathbb{H}^2\times\r$, J. Aust. Math. Soc. 85 (2008), 113--143.
 
 \bibitem{ET}
 R. S\'{a} Earp, E.~Toubiana.
 Screw motion surfaces in $\mathbb{H}^2\times\r$ and
$\mathbb{S}^2\times\r$, {\em Illinois J. Math.}  49 (2005), 1323--1362.

\bibitem{S} H.F. Scherk. Bemerkungen uber die kleinste Fl\"ache innerhalb gegebener Grenzen. {\em Journal f\"ur die reine und angewandte Mathematik} 13 (1835), 185--208. 

\bibitem{scott} P. Scott.  The geometries of $3$-manifolds. {\em Bull. London Math. Soc.} 15 (1983), 401--487.


%\bibitem{TO} P.~Tomter. Constant mean curvature surfaces in the Heisenberg group.
%{\em Proc. Sympos. Pure Math.} 54, 485--495, Amer. Math. Soc., Providence, RI, 1993.

%\bibitem{se}
%W. Seaman. Helicoids of constant mean curvature and their Gauss maps, {\em Pacific J. Math}. 110 (1984), 387--396. 

\bibitem{Vr}  G.~Vranceanu. {\em Le\c{c}ons de g\'{e}om\'{e}trie
diff\'{e}rentielle}.
Ed. Acad. Rep. Pop. Roum., vol. I, Bucarest, 1957.

\bibitem{W} W. Wunderlich, Beitrag zur Kentnis der Minimalscharaubflachen, {\em Compositio Math.}
(1952), 297--311.

\end{thebibliography}
\end{document}